\documentclass{amsart}
\usepackage{amsmath, amsthm, newlfont, amscd, amsfonts, amssymb, graphicx, color}
\usepackage[bookmarksnumbered, colorlinks, plainpages]{hyperref}
\newtheorem{theorem}{Theorem}[section]

\theoremstyle{definition}

\newtheorem{problem}[theorem]{Problem}

\numberwithin{equation}{section}

\begin{document}

\title[$\mathcal{Q}$-groups satisfy the equation $T_G(r,s)=0$]{$\mathcal{Q}$-groups satisfy the equation $T_G(r,s)=0$}

\author[F.G. Russo]{Francesco G. Russo}
\address{Department of Mathematics and Applied Mathematics \endgraf
University of Cape Town \endgraf
Private Bag X1, Rondebosch \endgraf
7701, Cape Town, South Africa}
\email{francescog.russo@yahoo.com}



\date{\today}

\begin{abstract} The present note shows that $\mathcal{Q}$-groups in [H. Heineken and F.G. Russo, Groups described by element numbers, Forum Math. 27 (2015), 1961--1977] are  solvable groups (not necessarily nilpotent) for which the equation $T_G(r,s)=0$ is satisfied.
\end{abstract}

\maketitle

\section{Two different approaches of investigation for the local breadth}

The present note deals with finite groups only. It is shown that the $\mathcal{Q}$-groups  in \cite[Page 1962]{hr1} and the groups of Meng and Shi in \cite{msbis, ms, mstris} satisfy the equation $T_G(r,s)=0$, mentioned in  \cite[Questions 1.4 and 1.5]{garonzi1}.

Chronologically \cite{garonzi1} appears two years after \cite{hr1}, but even \cite{tarny1, garonzi2, tarny2, tarny3} do not mention the so called   ``Inverse Problem to Frobenius' Theorem'' in the study of structural properties of groups by restrictions of numerical nature. A brief historical remark is then appropriate. Frobenius \cite{frobenius, frobeniusbis} and  Yamaki  \cite{yamaki1, yamaki2, yamaki3, yamaki4, yamaki5} worked on the following problem long time ago:  

\medskip
\medskip

\begin{problem}\label{frobenius}
Given  a group $G$ and  an integer $m \ge 1$ dividing the exponent $\exp(G)$, define  \begin{equation}\label{e1}L_m(G)=\{x \in G \ | \ x^m=1\}
\end{equation}
and study relations between $m$, $|L_m(G)|$ and the structure of $G$.  
\end{problem}

\medskip
\medskip

Frobenius showed that a divisor $m$ of $|G|$  divides also $|L_m(G)|$, that is, $|L_m(G)|=mk$ for some $k \ge 1$,  on the other hand, he conjectured that for $k=1$, then the $m$ elements of $L_m(G)$ form a characteristic subgroup of $G$. This was indeed shown by Iiyori and Yamaki in their aforementioned contributions with the help of the classification of simple groups. The elegance of the proof of Frobenius' Conjecture cannot be summarized here, but its importance is  clear since  no restriction on $G$ is given (of any nature !). A more recent problem, studied in  \cite{msbis, hr1, hr2, ms, mstris}, is to look at the size of $k$, and look for structural conditions on $G$. More precisely:

\medskip
\medskip

\begin{problem}[Inverse Problem to Frobenius' Theorem] \label{inversefrobenius} Classify all groups $G$ such that $m^{-1} \ |L_m(G)| \le \psi(m)$ for all $m$ dividing $\exp(G)$, where $\psi : m \in  \mathbb{N} \mapsto \psi(m) \in \mathbb{N}$ is a prescribed function depending only on  $m$.
\end{problem}

\medskip
\medskip

Now we change completely perspective and look numerically at the size of $L_m(G)$. Consider the Euler function $\varphi(k)$ and define the number
$$c_k=|\{C \le G \ | \ C \ \mbox{is a cyclic subgroup of order }  \ k \}|$$ 
of cyclic subgroups of $G$ of order exactly $k$. One can see that
\begin{equation}\label{e2}
 |L_m(G)|=   \sum_{k  \ | \ m} c_k \varphi(k) = m \cdot f(m),
\end{equation}
where $f(m) \ge 1$ is an integer depending only on $m$ and $G$.

Denoting by $r,s$ two real numbers and by $n= |G|$, Garonzi and Patassini \cite[Page 683, \S 2.4]{garonzi1} worked on the function
\begin{equation}\label{e3}
T_G(r,s)=\sum_{k \ | \ n} g_{k, n/k}^{r,s}  \ k \ (f(k)-1), 
\end{equation}
where $g_{k, \ n/k}^{r,s}$ are the coefficient expressed by the expansion  \cite[(1) of Lemma 7]{garonzi1}. These coefficients turn out to be  non-negative and the case $g_{k, \ n/k}^{r,s}=0$ is completely described by \cite[final part of (1) in Lemma 7]{garonzi1}.

Restrictions on $T_G(r,s)$ allows us to detect nilpotency in $G$ by \cite[Theorem 5]{garonzi1} and further  properties  were also noted by De Medts and  T\v{a}rn\v{a}uceanu in \cite{tarny1, tarny2, tarny3}. In particular, \cite[Questions 1.4]{garonzi1} asks whether the equality $T_G(r,s)=0$ detects solvability for $G$ for some $(r,s)$, while \cite[Questions 1.5]{garonzi1} asks whether the equality $T_G(r,s)=0$ implies structural properties for $G$ for some $(r,s) \in \mathbb{R}^2$.

\section{The main result and its proof}

In order to attack Problem \ref{inversefrobenius}, the authors of \cite{msbis, hr1, hr2, ms, mstris}  studied
what is called  the \textit{local breadth}, that is,  
\[\mathbf{b}_m(G) = |L_m(G)|  \ \cdot  \ m^{-1}\] of $G$ (under the general condition that $m$ is a divisor of the exponent of $|G|$, so not necessarily of $|G|$). In particular, we get
$$f(m)=\mathbf{b}_m(G) $$
comparing \eqref{e1} with \eqref{e2}.

 A first body of results is the following, where only nilpotent groups appear.

\begin{theorem}[See \cite{ms}, Main Theorem and \cite{hr2}, list in \S 4]\label{ms}
Let $G$ be a group and  $m \ge 1$ a divisor of $\exp(G)$. Then $f(m) \le 2$   if and only if 
\begin{itemize}   
\item[(i)] $G$ is cyclic;
\item[(ii)] $G \simeq C_h \times C_{2^{k-1}} \times C_2$ with $h$ odd and $k \ge 2$;
\item[(iii)] $G \simeq  C_h \times Q_8= C_h \times \langle a,b \ | \  a^4=1, a^2=b^2, b^{-1}ab=a^{-1}\rangle$ with $h$ odd;
\item[(iv)] $G \simeq C_h \times \langle a,b \ | \ a^{2^{t-1}}=b^2=1, b^{-1}ab=a^{1+2^{t-2}}\rangle$ with $t \ge 4$ and $h$ odd;
\item[(v)]$G \simeq C_h \times \langle a,b \ | \ a^3=b^{2^s}=1, b^{-1}ab=a^{-1}\rangle$ with $s \ge 1$ and $\gcd(h,6)=1$.
\end{itemize}
\end{theorem}

One can see without difficulties that the previous result remains true when $m$ is a divisor of $|G|$. After Theorem \ref{ms}, Chen, Meng and Shi \cite[Theorems 1.1, 1.2]{msbis} improved the classification based on the bound  $f(m) \le 2$ with another one, based on the bound $f(m) \le 3$. Successively it was introduced the more comprehensive  notion of $\mathcal{Q}$-$group$ in \cite{hr1} (i.e.:  a group $G$ such that  $f(m) \le m$ with $m$ divisor of  $\exp(G)$) and corresponding classifications are presented in   \cite[Theorems 3.2, 3.5, 3.6, 3.8, 3.12, 3.14]{hr1}, which we are going to summarize below.

\begin{theorem}[Structure of $\mathcal{Q}$-groups, see \cite{hr1}]\label{qgroups}
Let $G$ be a group and  $m \ge 1$ a divisor of $\exp(G)$ such that $f(m) \le m$.
\begin{itemize}   
\item[(i)] $G$ is solvable;
\item[(ii)] $G$ has a $2$-nilpotent normal subgroup $M$ of index $|G:M| \in  \{1, 3\}$;
\item[(iii)] if $p >3$ is a prime and $G$ a $p$-group, then $G$ is the product of two cyclic groups with trivial intersection;
\item[(iv)] if $\gcd(|G|,6)=1$, then $G$ is metabelian and the quotient $G/F(G)$ through the Fitting subgroup $F(G)$ is cyclic;
\item[(v)]$G=\mathrm{SL}(2,3)$ is a solvable non-metabelian  $\mathcal{Q}$-group with $\gcd(|G|,6) \neq 1$.
\end{itemize}
\end{theorem}
 
The notation for the Fitting subgroup, the notion of $p$-nilpotent group and the notation for the special linear group of dimension $2$ over the field with $3$ elements are very common and recalled in most of the references in the bibliography. Thanks to Theorem \ref{qgroups}, we are in the position to show that a large class of solvable groups, not necessarily nilpotent, satisfies the equation $T_G(r,s)=0$.

\begin{theorem}\label{t1}If $(r,s) \in \mathbb{R}^2$  satisfy  $1 \le g_{k, \ n/k}^{r,s}$  and $f(k) \le  {\big( g_{k, \ n/k}^{r,s}\big)}^{-\frac{1}{2}} \le k$   for all   $k$ divisors of $n$,  then  $T_G(r,s)=0$ and $G$ is solvable.
\end{theorem}

\begin{proof}[Proof of Theorem \ref{t1}]
It is useful to  note that \eqref{e3} becomes equal to zero when $f(k)=1$ and $f(k)=1$ happens if and only if $G$ satisfies (i) of Theorem \ref{ms}, that is,  $G$ is cyclic. Therefore the real issue of studying $T_G(r,s)=0$ is when $f(k) \geq 2$, so there is no loss of generality in assuming  $f(k) \ge 2$. Since
all $g_{k, \ n/k}^{r,s} \ge 1$, we have  $T_G(r,s) \ge 0$ and it is enough to show $T_G(r,s) \le 0$ in order to conclude  that $T_G(r,s)=0$. Now 
$$f(k) \le  {\big( g_{k, \ n/k}^{r,s}\big)}^{-\frac{1}{2}} \ \ \Rightarrow \ \ g_{k, n/k}^{r,s}   \  f(k) \le \sqrt{g_{k, n/k}^{r,s}}   $$
implies 
$$  \sum_{k \ | \ n} g_{k, n/k}^{r,s}   \  f(k)   \ k  \le  \sum_{k \ | \ n} \sqrt{g_{k, n/k}^{r,s}}      \ k \le  \sum_{k \ | \ n}  g_{k, \ n/k}^{r,s}  \ k $$ 
and so
$$  \sum_{k \ | \ n} g_{k, n/k}^{r,s}  \ k \ (f(k)-1) \le 0  $$ 
which gives what we claimed $T_G(r,s)=0$. On the othe hand,  $G$ is a $\mathcal{Q}$-group because $f(k) \le k$ and so it is solvable by (i) of Theorem \ref{qgroups}. It follows that there exist $(r,s) \in \mathbb{R}^2$ such that $T_G(r,s)=0$ and $G$ is solvable.
        \end{proof}

While the existence of $(r,s)$ is shown under the conditions of Theorem \ref{t1}, the structural conditions of $G$ are elucidated in Theorem \ref{qgroups}.


\begin{thebibliography}{20}



\bibitem{msbis} K.Chen, W. Meng and J. Shi, On an inverse problem to Frobenius' theorem II, \textit{J. Algebra Appl.} \textbf{11} (2012), Paper ID: 1250092 [8 pages]. 

\bibitem{tarny1}T. De Medts and  M. T\v{a}rn\v{a}uceanu, Finite groups determined by an inequality of the orders of their subgroups, \textit{Bull. Belgian. Math. Soc. Simon Stevin} \textbf{15} (2008), 699--704.

\bibitem{frobenius} G. Frobenius, Verallgemeinerung des Sylowschen Satzen, \textit{Berliner Sitz.} (1895), 981--993.


\bibitem{frobeniusbis} G. Frobenius, \"Uber  einen Fundamentalsatz der Gruppentheorie, \textit{Berliner Sitz.} (1903), 987--991. 

\bibitem{garonzi1}M. Garonzi and M. Patassini, Inequalities detecting structural properties of a finite group, \textit{Comm. Algebra} \textbf{45} (2017), 677--687.

\bibitem{garonzi2}M. Garonzi and I. Lima, On the number of cyclic subgroups of a finite group, \textit{Bull. Brazilian Math. Soc.} \textbf{49} (2018),  515--530.


\bibitem{hr1}H. Heineken and F.G. Russo, Groups described by element numbers, \textit{Forum Math.} \textbf{27} (2015), 1961--1977.


\bibitem{hr2}H. Heineken and F.G. Russo, On a notion of breadth in the sense of Frobenius, \textit{J. Algebra} \textbf{424} (2015), 208--221.



\bibitem{yamaki1}N. Iiyori and H. Yamaki, On a conjecture of Frobenius, \textit{Bull. Amer. Math. Soc.} \textbf{25} (1991), 413--416.

\bibitem{ms} W. Meng and J. Shi, On an inverse problem to Frobenius' theorem, \textit{Arch. Math. (Basel)} \textbf{96} (2011), 109--114.

\bibitem{mstris} W. Meng, Finite groups of global breadth four in the sense of Frobenius, \textit{Comm. Algebra} \textbf{45} (2016), 660--665.


\bibitem{tarny2}   M. T\v{a}rn\v{a}uceanu, An arithmetic method of counting the subgroups of a finite abelian group, \textit{Bull. Math. Soc. Sci. Math. Roumanie} \textbf{53} (2010), 373--386.



\bibitem{tarny3}   M. T\v{a}rn\v{a}uceanu, A nilpotency criterion for finite groups, \textit{Acta Math. Hungarica}
\textbf{15} (2018),  499--501.

\bibitem{yamaki2} H. Yamaki, A conjecture of Frobenius and the sporadic simple groups I, \textit{Comm. Algebra} \textbf{11} (1983), 2513--2518.

\bibitem{yamaki3} H. Yamaki, A conjecture of Frobenius and the simple groups of Lie type I, \textit{Arch. Math. (Basel)} \textbf{42} (1984), 344--347.

\bibitem{yamaki4} H. Yamaki, A conjecture of Frobenius and the simple groups of Lie type II, \textit{J. Algebra} \textbf{96} (1985), 391--396.

\bibitem{yamaki5} H. Yamaki, A conjecture of Frobenius and the sporadic simple groups II, \textit{Math. Comp.} \textbf{46} (1986), 609--611.





\end{thebibliography}
\end{document}